\documentclass[a4paper, reqno]{amsart}

\usepackage[margin=1.5in]{geometry}
\usepackage{amssymb, latexsym, amsmath}
\usepackage{stmaryrd}
\usepackage[none]{hyphenat}
\usepackage{hyperref}
\usepackage{enumitem}
\usepackage{float}
\usepackage{makecell}
\usepackage[subnum]{cases}
\usepackage{multirow}
\usepackage{tikz}
\usetikzlibrary{arrows, arrows.meta}
\usepackage{thm-restate}

\DeclareMathOperator{\rank}{rank}

\newtheorem*{mainresult}{Main Result}

\newtheorem{proposition}{Proposition}
\newtheorem{lemma}{Lemma}

\theoremstyle{definition}
\newtheorem{remark}{Remark}

\newcommand\Bstrut{\rule[-1.8ex]{0pt}{0pt}}
\newcommand\Tstrut{\rule{0pt}{3.2ex}}

\newcommand{\tor}[1]{#1_{\text{tors}}}

\title[Constructing congruent number elliptic curves using 2-descent]{Constructing congruent number elliptic curves using 2-descent}
\author{Raiza Corpuz}
\date{}

\begin{document}

\begin{abstract}
A positive integer that is the area of some rational right triangle is called a 
congruent number. In an algebraic point of view, being a congruent number 
means satisfying a system of equations. As early as the 1800s, it is understood 
that if $n$ is a congruent number, then the equation $nm^2 = uv(u^2 - v^2)$ 
has a solution in $\mathbb{Z}$. Using the relation between congruent numbers 
and elliptic curves $E_n: y^2 = x^3 - n^2 x$ which was established in the 1900s, 
we will prove that the converse of this two century-old result holds as well. In 
addition to this, we present another proof of the converse using the method of 
2-descent. Towards the end of this paper, we demonstrate how one can use our 
proof to construct subfamilies of $E_n$ with rank at least 2 and 3.  
\end{abstract}
\maketitle

\section{Introduction}
\noindent A positive integer $n$ is said to be a \emph{congruent number} 
if there exists a right triangle with rational sides and area $n$. 
The search for congruent numbers is more commonly referred to as 
the \emph{congruent number problem}. This curiosity dates back to 
200\,AD, around the time of the Greek mathematician Diophantus. 
Meanwhile, an elliptic curve is an object that also has a rich history. 
It has been implicitly studied as Diophantine equations and a millenium 
passed before it was finally formalized.

An \emph{elliptic curve $E$ over a number field $K$} denoted $E/K$ is 
a nonsingular cubic projective curve with affine equation 
\[ Y^2 = X^3 + AX + B, \]
where $A, B \in K$. If we introduce the notion of adding points by the 
secant and tangent method, the set $E(K)$ of $K$-rational points of $E$, 
with a unique point at infinity $\mathcal{O}$ taken as the identity element, 
forms a group. In 
1922, Mordell proved that $E(\mathbb{Q})$ is a finitely generated 
abelian group. A few years later, Weil extended this result over 
arbitrary number fields giving us the Mordell-Weil 
theorem that we have presently: 
\[ E(K) \cong \tor{E(K)} \oplus \mathbb{Z}^{r}. \] 
The \emph{torsion subgroup} $\tor{E(K)}$ consists of the points 
of finite order in $E(K)$ and the number $r$ is called the 
\emph{free rank} of $E(K)$.

The 19th century was a promising time in the development of the 
congruent number problem. It was the time when the connection 
between congruent numbers and elliptic curves surfaced.  This 
relation is made more explicit by the following propositon 
(cf.\ \cite{koblitz}, Prop.\ 18, p.\ 46).

\begin{proposition}\label{sec1prop1}
A positive square-free integer $n$ is a congruent number if and only 
if the elliptic curve 
\[ E_{n}: y^2 = x \left( x^2 - n^2 \right) \]
has positive rank. If $n$ is a congruent number, we call $E_{n}$ the 
congruent number elliptic curve corresponding to $n$.
\end{proposition}

Proposition \ref{sec1prop1} offers a way to check whether a given 
positive number is a congruent number without having to draw rational 
right triangles. However, it is not a practical tool to use to obtain the  
set of all congruent numbers, as the rank of elliptic curves proved to be 
nontrivial to compute. But if one can find an expression for $n$ and show 
that $y^2 = x^3 - n^2 x$ has positive rank, then one can generate a subset 
of the set of all congruent numbers. For instance, Johnstone and Spearman
\cite{johnstonespearman1} considered 
$n = 6 \left( r^4 + 2r^2s^2 + 4s^4 \right) \left( r^4 + 8r^2s^2 + 4s^4 \right)$ 
and proved that the family of elliptic curves $E/\mathbb{Q}: y^2 = x^3 - n^2 x$ 
has rank at least 3. They achieved a rank more than 1 by choosing relatively 
prime integers $r$ and $s$  so that $t := \frac{r}{s}$ satisfies 
$w^2 = t^4 + 14t^2 + 4$. There are infinitely many such pairs $(r,s)$ 
because the quartic curve they considered is birationally equivalent to 
an elliptic curve with positive rank. 

Meanwhile, Bennett\cite{bennett} did not start with an expression for $n$ but 
instead proved that if $n$ is a positive integer so that the Diophantine equation
\[ x(x + 1)(x + 2) = ny^2 \]
has a solution, then the rank of $E_{n}$ is positive. He then used this 
to show that for any integer $m > 1$, there are infinitely many
congruent numbers in each congruence class modulo $m$.
Johnstone and Spearman\cite{johnstonespearman2} extended this 
result by considering $n = a^4 - b^4$, where 
\[ (a,b) = \left( \dfrac{3t^2 + 6t + 1}{2}, \dfrac{3t^2 + 2t + 1}{2} \right). \]
The corresponding family of congruent number elliptic curves 
has rank at least 2. Similar to \cite{bennett}, they proceeded to show that 
for any integer $m > 1$, each congruence class modulo $m$ contains infinitely 
congruent numbers for which the corresponding congruent number elliptic 
curves belong to their constructed family, so they have rank at least 2. 
In 2009, Dujella, Janfada, and Salami\cite{djs} developed an algorithm to 
search for congruent number elliptic curves of high rank based on Monsky's 
formula for computing the 2-Selmer rank (see \cite{silverman}, chapter X)
of congruent number elliptic curves. According to this study, the highest 
observed rank for congruent numbers is 7 which is due to Rogers\cite{rogers}.

More than a century ago, Roberts\cite{roberts} proved that if $n$ is a 
congruent number. then one can write
\begin{equation}\label{exp}
	nm^2 = uv(u^2 - v^2),
\end{equation}
so that $m, u, v \in \mathbb{Z}$. To illustrate, let us take 
$n = 6 \left( r^4 + 2r^2s^2 + 4s^4 \right) \left( r^4 + 8r^2s^2 + 4s^4 \right)$ 
in \cite{johnstonespearman1}. If we let $u = r^4 + 2r^2s^2 + 4s^4$, 
$s = 6r^2s^2$, and $m = 1$ we see that indeed, we can write $n$ as 
$nm^2 = uv\left( u^2 - v^2 \right)$. Another example is when $n = a^4 - b^4$ 
in \cite{johnstonespearman2}. Here, we can take $u = a^2$, $v = b^2$, and 
$m = ab$. We also see that the converse of Roberts' result may also hold for some 
cases. An example is \cite{bennett}, where we may take $u = x + 2$, $v = x$, and 
$m = 2$.

Proposition \ref{sec1prop1} is a witness to the connection between congruent 
numbers and elliptic curves while Roberts' result exposes the connection between 
congruent numbers and the equation \eqref{exp}. One objective of this paper is 
to show the relation between equation \eqref{exp} and congruent number elliptic 
curves. This relation is summarized below. 

\begin{mainresult}
Fix $n \in \mathbb{Z} \setminus\{0\}$. The elliptic curve $y^2 = x^3 - n^2 x$ 
has positive rank over $\mathbb{Q}$ if and only if we can write
\[ nm^2 = uv(u^2 - v^2), \]
for some integers $u$ and $v$.
\end{mainresult}

\noindent The second objective is to use the method of 2-descent to give a proof 
of the reverse direction and demonstrate how this method of proof serves as 
foundation for constructing congruent number elliptic curves of rank more than 1.

We give a brief overview of the paper. In the next section, we will discuss how the 
congruent number problem was perceived from the time of Diophantus up to the 
present. In particular, we will derive some sytems of equations that were used in 
an attempt to provide an answer to this millenium-old problem algebraically. In 
Section 3, we discuss the method of 2-descent, which we will then use in Section 
4, to prove the converse of our main result. In the succeeding section, we will 
construct families of congruent number elliptic curves with rank at least 2 and 3 
from the results laid out in Section 4.

%%%%%%%%%%%%%%%%%%%%%%%%%%%%%%%%%%%%%%%%%%%%%

\section{System of equations relevant to the congruent number problem}

\noindent As mentioned before, the congruent number problem dates back to 
the Greeks who asked which numbers occur as the area of some rational right 
triangle. Suppose $n$ is a congruent number. Algebraically speaking, this means 
that there is a triple $(X, Y, Z) \in \mathbb{Q}^3$ so that
\begin{equation}\label{eq1}
\begin{cases}
	\begin{aligned}
			X^2 + Y^2 &= Z^2\\
			XY &= 2n.
	\end{aligned}
\end{cases}
\end{equation}
From this system, one can prove that $n$ is a congruent number if and only if 
$\alpha^2 n$, $\alpha \in \mathbb{Z} \setminus \{0\}$ is a congruent 
number as well. If $n$ is a congruent number and $(x, y, z)$ is a solution to 
\eqref{eq1}, then $(\alpha x, \alpha y, \alpha z)$ is a solution when $n$ is 
replaced by $\alpha^2 n$. On the other hand if $\alpha^2 n$ is a congruent 
number and $(x', y', z')$ is a solution to \eqref{eq1}, then one can consider 
the triple $\left( \tfrac{x}{\alpha}, \tfrac{y}{\alpha}, \tfrac{z}{\alpha} \right)$ 
to show that $n$ is a congruent number. We formalize this result in the following 
proposition. 

\begin{proposition}\label{sec2prop1}
A positive integer $n$ is a congruent number if and only if $\alpha^2 n$, for 
some $\alpha \in \mathbb{Z} \setminus \{0\}$ is a congruent number.
\end{proposition}

Therefore, in the study of congruent numbers, it is enough to look at their 
squarefree parts. And so for this section, we shall take $n$ to be a nonzero 
squarefree natural number. We take another look at system \eqref{eq1}. 
Notice that it can be rewritten as: 
\begin{equation}\label{eq2}
	(X \pm Y)^2 = Z^2 \pm 4n \iff \left( \dfrac{X \pm Y}{2} \right)^2 = \left( \dfrac{Z}{2} \right)^2 \pm n.
\end{equation}
Let $W$ be the least common denominator of $X$, $Y$, and $Z$. Multiplying 
both sides by $4W$, we get
\begin{equation}\label{eq3}
	 \left( X \pm Y \right)^2W^2 = (ZW)^2 \pm n(2W)^2.
\end{equation}

An Arab manuscript witten around the 10th century suggests that they were 
also studying the congruent number problem. Although they envisioned the 
problem in another way: they were particularly interested in the existence 
of three squares in arithmetic progression and common difference $n$. 
Algebraically, this question asks if there is a triple $(P, R, S)$ satisfying the 
following system of Diophantine equations.
\begin{equation}\label{eq4.0}
\begin{cases}
	\begin{aligned}
			R^2 - P^2 &= n\\
			P^2 - S^2 &= n
	\end{aligned}
\end{cases}
\end{equation}
It can easily be seen that if $(P, R, S)$ satisfies the system \eqref{eq4.0}, then 
$(\alpha P, \alpha R, \alpha S)$ satisfies the same equation when $n$ is 
replaced by $\alpha^2 n$, for some nonzero integer $\alpha$. One can see that 
the Diophantine equations \eqref{eq4.0} that describe the arithmetic progression 
of three squares problem can be generalized into the search for a 4-tuple 
$(P, Q, R, S) \in \mathbb{N}^4$ so that
\begin{equation}\label{eq4.1}
\begin{cases}
	\begin{aligned}
			P^2 + nQ^2 = R^2\\
			P^2 - nQ^2 = S^2.
	\end{aligned}
\end{cases}
\end{equation}
With $(P, Q, R, S) = \left( ZW, 2W, (X + Y)W, (X - Y)W \right) \in \mathbb{Z}^4$, 
we see that system \eqref{eq4.1} follows from \eqref{eq3}.
To prove the other way, subtract the equations in \eqref{eq4.1} 
appropriately to get
\[ n = \dfrac{1}{2Q^2}\left( R^2 - S^2 \right) = \dfrac{1}{2}\left( \dfrac{R + S}{Q} \right)\left( \dfrac{R - S}{Q} \right) \]
and then show that $\tfrac{R + S}{Q}, \tfrac{R - S}{Q} \in \mathbb{Q}$ 
are sides of a right triangle. Indeed, since they satisfy the 
Pythagorean formula
\[ \left( \dfrac{R + S}{Q} \right)^2 + \left( \dfrac{R - S}{Q} \right)^2 = \dfrac{2\left( R^2 + S^2 \right)}{Q^2} \stackrel{\eqref{eq4.1}}{=} \dfrac{2(2P^2)}{Q^2} = \left( \dfrac{2P}{Q} \right)^2 \]
they represent the base and the height of a rational right triangle. 
This shows that the Greeks (\eqref{eq1} $\iff$ \eqref{eq3}) and the Arabs 
\eqref{eq4.1} were indeed studying the same problem, although interpreted 
differently. An interesting observation is that as equivalent systems of equations, 
\eqref{eq1} and \eqref{eq4.1} can be freely extended to the case when $n$ is a 
negative integer. It loses its geometric meaning but it remains algebraically sound. 
We will discuss this further in Section 4. 

In 1879, while trying to answer a problem of 
Fibonacci, Roberts\cite{roberts} proved the following result. 

\begin{proposition}\label{sec2prop2}
Fix $n \in \mathbb{Z} \setminus\{0\}$. If the system of Diophantine equations 
\eqref{eq4.1} is solvable, then we can write 
\[ nm^2 = uv(u^2 - v^2). \]
for some integers $u$, $v$, and $m$.
\end{proposition}
\begin{proof}
Our proof will utilize what we know about congruent numbers so far. Let $n$ 
be a congruent number. Using the equivalence of systems \eqref{eq1} and 
\eqref{eq4.1}, we can say that there is a triple $(x, y, z)$ of positive rational 
numbers satisfying \eqref{eq1}. In particular, write $x = \tfrac{x_1}{x_2}$, 
$y = \tfrac{y_1}{y_2}$, and $z = \tfrac{z_1}{z_2}$ as a ratio of positive 
integers (in lowest terms). By clearing denominators, \eqref{eq1} transforms 
into
\begin{equation}\label{eq5}
\begin{cases}
	\begin{aligned}
			(x_1 y_2 z_2)^2 + (x_2 y_1 z_2)^2 &= (x_2 y_2 z_1)^2\\
			x_1 y_1 &= 2n x_2 y_2.
	\end{aligned}
\end{cases}
\end{equation}
The first equation in the above system has a complete known parametrization: 
\begin{equation}
\begin{cases}
	\begin{aligned}
			x_1 y_2 z_2 &= k\left( r^2 - s^2 \right)\\
			x_2 y_1 z_2 &= 2krs\\
			x_2 y_2 z_1 &= k \left( r^2 + s^2 \right),
	\end{aligned}
\end{cases}
\end{equation}
with $u,v \in \mathbb{Z} \setminus \{0\}$ and 
$k = \gcd\left( x_1 y_2 z_2, x_2 y_1 z_2, x_2 y_2 z_1 \right)$. Rewrite the 
second equation in \eqref{eq5} using this parametrization to get: 
\[ x_1 y_1 = 2n x_2 y_2 \iff x_1 y_1 x_2 y_2 z_2^2 = 2n(x_2 y_2 z_2)^2 \iff k^2 rs\left( r^2 - s^2 \right) = n \left( x_2 y_2 z_2 \right)^2. \]
Multiply both sides by $k^2$ to get 
$k^4 rs \left( r^2 - s^2 \right) = n \left( kx_2 y_2 z_2 \right)^2$. Finally, 
letting $u = kr$, $v = ks$, and $m = kx_2 y_2 z_2$, we have 
$nm^2 = uv \left( u^2 - v^2 \right)$as desired.
\end{proof}

The most recent development for the congruent number problem came in 1983, 
when Tunnell\cite{tunnell} developed a criterion to check if a positive integer 
$n$ is a congruent number. It involves counting the solutions of several 
Diophantine equations. Let 
\begin{align*}
	A_n &= \#\left\{ (x,y,z) \in \mathbb{Z}^3: n = 2x^2 + y^2 + 32z^2 \right\}\\
	B_n &= \#\left\{ (x,y,z) \in \mathbb{Z}^3: n = 2x^2 + y^2 + 8z^2 \right\}\\
	C_n &= \#\left\{ (x,y,z) \in \mathbb{Z}^3: n = 8x^2 + 2y^2 + 64z^2 \right\}\\
	D_n &= \#\left\{ (x,y,z) \in \mathbb{Z}^3: n = 8x^2 + 2y^2 + 16z^2 \right\}.
\end{align*}
Tunnell's criterion says that for an odd congruent number $n$ we have $2A_n = B_n$. 
On the other hand, if $n$ is an even congruent number, then $2C_n = D_n$.

%%%%%%%%%%%%%%%%%%%%%%%%%%%%%%%%%%%%%%%%%%%%%

\section{The Method of 2-Descent}

\noindent As mentioned before, the computation of a rank of an arbitrary 
elliptic curve is no trivial task. In fact, the only method so far is a well-known 
conjecture known as the Birch and Swinnerton-Dyer (BSD) conjecture. In fact, 
if this conjecture turns out to be true, then Tunnell's criterion is a necessary 
and sufficient condition for a positive integer $n$ to be a congruent number. 
To learn more about this conjecture one can check Appendix C.16 of 
\cite{silverman}. 

Fortunately, for some elliptic curves, there is a method called 2-Descent that 
we can use. In what follows, we discuss the method of 2-Descent applied to 
elliptic curves with a rational point of order two, as in Section 3.4 of 
\cite{silvermantate}.

Let $A$ be a nonzero integer. Take the elliptic curve
\[ E/\mathbb{Q}: y^2 = x \left( x^2 + A \right) \]
and set $T = (0,0) \in E(\mathbb{Q})$. Consider another elliptic curve 
\[ \overline{E}/\mathbb{Q}: y^2 = x^3 - 4Ax, \]
and set $\overline{T} = (0,0) \in \overline{E}(\mathbb{Q})$. 
The existence of the following homomorphisms tells us that the elliptic curves 
$E$ and $\overline{E}$ are isogenous over $\mathbb{Q}$.
\begin{enumerate}[label = (\roman*)]
	\item{$\phi: E(\mathbb{Q}) \to \overline{E}(\mathbb{Q})$ defined as
	\[ \phi(P) = 
		\begin{cases}
			\left( \dfrac{y^2}{x^2}, \dfrac{y \left( x^2 - A \right)}{x^2} \right), &\text{ if } P = (x,y) \neq \mathcal{O}, T,\\
			\ \ \ \ \ \ \ \ \ \ \overline{\mathcal{O}}, &\text{ if } P = \mathcal{O} \text{ or } P = T,
		\end{cases} \]
	with $\ker(\phi) = \{ \mathcal{O}, T \}$.} 
	\item{$\psi: \overline{E}(\mathbb{Q}) \to E(\mathbb{Q})$ defined as 
	\[ \psi(P) = 
		\begin{cases}
			\left( \dfrac{\overline{y}^2}{4 \overline{x}^2}, \dfrac{\overline{y} \left( \overline{x}^2 + 4A \right)}{8 \overline{x}^2} \right), &\text{ if } \overline{P} = ( \overline{x}, \overline{y} ) \neq \overline{\mathcal{O}}, \overline{T},\\
			\ \ \ \ \ \ \ \ \ \ \mathcal{O}, &\text{ if } \overline{P} = \overline{\mathcal{O}} \text{ or } \overline{P} = \overline{T},
		\end{cases} \]
	with $\ker(\psi) = \{ \overline{\mathcal{O}}, \overline{T} \}$.}
\end{enumerate}
For any $P \in E(\mathbb{Q})$ and 
$\overline{P} \in \overline{E}(\mathbb{Q})$, the above homomorphisms 
satisfy $\left( \psi \circ \phi \right)(P) = 2P$ and 
$\left( \phi \circ \psi \right)\left( \overline{P} \right)= 2\overline{P}$. 
We introduce two more homomorphisms:
\begin{enumerate}[label = (\roman*)]
	\item{$\alpha: E(\mathbb{Q}) \to \mathbb{Q}^{\times}/(\mathbb{Q}^{\times})^{2}$ defined as
	$$ \alpha(P) = 
		\begin{cases}
			\begin{aligned}
				1 \left( \bmod\ (\mathbb{Q}^{\times})^{2} \right), &\text{ if } P = \mathcal{O},\\
				A \left( \bmod\ (\mathbb{Q}^{\times})^{2} \right), &\text{ if } P = T,\\
				x \left( \bmod\ (\mathbb{Q}^{\times})^{2} \right), &\text{ if } P = (x,y), \text{ with } x \neq 0.
			\end{aligned}
		\end{cases}$$}
	\item{$\overline{\alpha}: \overline{E}(\mathbb{Q}) \to \mathbb{Q}^{\times}/(\mathbb{Q}^{\times})^{2}$ given by
	$$ \overline{\alpha}(P) = 
		\begin{cases}
			\begin{aligned}
				1 \left( \bmod\ (\mathbb{Q}^{\times})^{2} \right), &\text{ if } \overline{P} = \overline{\mathcal{O}},\\
				-A \left( \bmod\ (\mathbb{Q}^{\times})^{2} \right), &\text{ if } \overline{P} = \overline{T},\\
				\overline{x} \left( \bmod\ (\mathbb{Q}^{\times})^{2} \right), &\text{ if } \overline{P} = \left( \overline{x}, \overline{y} \right), \text{ with } \overline{x} \neq 0.
			\end{aligned}
		\end{cases}$$}
\end{enumerate}
These homomorphisms are related via the following isomorphisms:
\begin{enumerate}[label = ($\roman*$)]
	\item{$\alpha(E(\mathbb{Q})) \cong E(\mathbb{Q})/\psi \left(\overline{E}(\mathbb{Q}) \right)$}
	\item{$\overline{\alpha}\left( \overline{E}(\mathbb{Q}) \right) \cong \overline{E}(\mathbb{Q})/\phi(E(\mathbb{Q}))$.}
\end{enumerate}

\noindent The following proposition shows us how these groups 
can help us study the rank of the elliptic curve $E$.

\begin{proposition}\label{sec3prop1}
Let $\alpha$ and $\overline{\alpha}$ be the homomorphisms described 
above. Suppose $A$ is a nonzero integer. If $r$ denotes the rank of the 
elliptic curve $E: y^2 = x \left( x^2 + A \right)$ over $\mathbb{Q}$, then
\begin{equation}\label{rankformula}
	2^{r} = \dfrac{\#\alpha\left( E(\mathbb{Q}) \right) \cdot \#\overline{\alpha}\left( \overline{E}(\mathbb{Q}) \right)}{4}.
\end{equation}
\end{proposition}

For a proof of this proposition, see Section 3.6 of \cite{silvermantate}. 
Now the above formula is only useful once we determine how to compute 
for the cardinality of the groups $\alpha\left( E(\mathbb{Q}) \right)$ 
and $\overline{\alpha}\left( \overline{E}(\mathbb{Q}) \right)$. But 
this task turns out to be a number-theoretic problem described as 
follows. 

\begin{enumerate}[label = (\roman*)]
		\item{$\#\alpha\left( E(\mathbb{Q}) \right)$ is the number of 
		distinct squarefree integers $b_{1}$ satisfying 
		$b_{1} \cdot b_{2} = A$, for which the Diophantine equation
		\begin{equation}
			N^2 = b_{1}M^4 + b_{2}e^4, \label{count1}
		\end{equation}
		has a nontrivial solution $(N, e, M)$ and 
		$$ \gcd(M,e) = \gcd(N,e) = \gcd\left( b_{1}, e \right) = \gcd\left( b_{2}, e \right) = \gcd(M,N) = 1. $$}
		\item{$\#\overline{\alpha}\left( \overline{E}(\mathbb{Q}) \right)$ 
		is the number of distinct squarefree integers 
		$\overline{b_{1}}$ satisfying $\overline{b_{1}} \cdot \overline{b_{2}} = -4A$, 
		for which the Diophantine equation
		\begin{equation}
			N^2 = \overline{b_{1}}M^4 + \overline{b_{2}}e^4, \label{count2}
		\end{equation}
		has a nontrivial solution $(N, e, M)$ and 
		$$ \gcd(M,e) = \gcd(N,e) = \gcd\left( \overline{b_{1}}, e \right) = \gcd\left( \overline{b_{2}}, e \right) = \gcd(M,N) = 1. $$}
\end{enumerate}

For the rest of the paper, we write $\mathcal{T}\left( b_{1} \right)$ 
to refer to  equation \eqref{count1} and 
$\overline{\mathcal{T}}\left( \overline{b_{1}} \right)$ to refer to 
equation \eqref{count2}. 

\begin{remark}\label{sec3rem1}
The group $\alpha\left( E(\mathbb{Q}) \right)$ is multiplicative. So if 
$b_{1}$ and $b'_{1}$ are divisors of $A$ such that 
$\mathcal{T}\left( b_{1} \right)$ and $\mathcal{T}\left( b'_{1} \right)$ 
have nontrivial solutions, meaning $b_1, b'_{1} \in \alpha(E(\mathbb{Q}))$, 
then the product $b_{1} b'_{1} \bmod\ (\mathbb{Q}^{\times})^2$ is also 
an element of $\alpha\left( E(\mathbb{Q}) \right)$. The same is true for the 
multiplicative group $\overline{\alpha}\left( \overline{E}(\mathbb{Q}) \right)$.
\end{remark}

\begin{proposition}\label{eltstopts}
If $b_1$ is an integer so that $\mathcal{T}(b_1)$ has a solution, then the 
corresponding quadruple $(b_1, N, e, M)$ gives rise to a point 
\[ \left( \dfrac{b_1M^2}{e^2}, \dfrac{b_1 MN}{e^3} \right) \text{ in } E(\mathbb{Q}). \]
Similarly, If $\overline{b_1}$ is an integer so that 
$\overline{\mathcal{T}}(\overline{b_1})$ has a solution, then the corresponding 
quadruple $(\overline{b_1}, N, e, M)$ gives rise to a point 
\[ \left( \dfrac{\overline{b_1}M^2}{e^2}, \dfrac{\overline{b_1} MN}{e^3} \right) \text{ in } \overline{E}(\mathbb{Q}). \]
\end{proposition}

Most of the time we will be interested in the elliptic curve $E/\mathbb{Q}$. 
If we get a point $P \in \overline{E}(\mathbb{Q})$, we can use the 
homomorphism $\psi$ to get the point $\psi(P) \in E(\mathbb{Q})$.

\section{An expression for congruent numbers}

\noindent In Section 2, we saw that a positive integer $n$ is a congruent number 
if and only if the equivalent system of equations \eqref{eq1} and \eqref{eq4.1} 
are solvable. At the end of the discussion, we posited considering any nonzero 
integer $n$ for \eqref{eq1} and \eqref{eq4.1}. But this meant letting go of the 
geometric connotation (\emph{i.e.} rational right triangles, three squares in 
arithmetic progression) and reducing the problem to determining the solvability 
of systems of equations. The idea to move the discussion to the ring of integers 
stems from the relationship between congruent numbers $n$ and elliptic curves 
of the form $E: y^2 = x^3 - n^2 x$. Notice how the elliptic curve does not 
discriminate between positive and negative integers. As a matter of fact, 
Proposition \ref{sec1prop1}and Proposition \ref{sec2prop1} can be refocused to 
include negative integers $n$. 

Let $n > 0$. One proof of Proposition \ref{sec1prop1} makes use of one-to-one 
correspondence between the sets $A$ and $B$, where 
\[ A[n] = \left\{ (a, b, c) \in \mathbb{Q}^3: a^2 + b^2 = c^2, ab = 2n, abc \neq 0 \right\}, \]
\[ B[n] = \left\{ (x,y) \in \mathbb{Q}^2: y^2 = x^3 - n^2 x, y \neq 0. \right\}. \]
Here, 
$\tor{E(\mathbb{Q})} \cong \mathbb{Z}/2\mathbb{Z} \times \mathbb{Z}/2\mathbb{Z}$, 
meaning a point of finite order in $E(\mathbb{Q})$ has order either one or two. In fact, 
$\tor{E(\mathbb{Q})} = \{ \mathcal{O}, (0, 0), (\pm n, 0) \}$. This tells us that $B[n]$ 
is the set of all points of $E(\mathbb{Q})$ of infinite order. Now the correspondence 
between the sets $A$ and $B$ is given by 
\[ \varphi: (a, b, c) \mapsto \left( \dfrac{nb}{c - a}, \dfrac{2n^2}{c - a} \right), \]
\[ \psi: (x, y) \mapsto \left( \dfrac{x^2 - n^2}{y}, \dfrac{2nx}{y}, \dfrac{x^2 + n^2}{y} \right). \]

Let $m$ be a positive integer. It is easy to see that $B[m] = B[-m]$ and 
$A[m] = A[-m]$. We show that there is also a one-to-one correspondence 
between $A$ and $B$ when $n < 0$. 

%\pagebreak

\noindent Let $a, b, c \in \mathbb{Q}^+$ and $x, y \in \mathbb{Q}$. 

\begin{center}
\begin{tabular}{c | c | c}
\hline
 & $n = m$ & $n = -m$\\ \hline
 {\multirow{4}{*}{$\varphi(A[n])$}} & $(a, b, c) \mapsto \left( \tfrac{mb}{c - a}, \tfrac{2m^2}{c - a} \right)$ & $(a, -b, c) \mapsto \left( \tfrac{mb}{c - a}, \tfrac{2m^2}{c - a} \right)$\Tstrut\Bstrut\\ \cline{2 - 3}
	 & $(-a, -b, -c) \mapsto \left( \tfrac{mb}{c - a}, -\tfrac{2m^2}{c - a} \right)$ & $(-a, b, -c) \mapsto \left( \tfrac{mb}{c - a}, -\tfrac{2m^2}{c - a} \right)$\Tstrut\Bstrut\\ \cline{2 - 3}
	 & $(-a, -b, c) \mapsto \left( -\tfrac{mb}{c + a}, \tfrac{2m^2}{c + a} \right)$ & $(-a, b, c) \mapsto \left( -\tfrac{mb}{c + a}, \tfrac{2m^2}{c + a} \right)$\Tstrut\Bstrut\\ \cline{2 - 3}
	 & $(a, b, -c) \mapsto \left( -\tfrac{mb}{c + a}, -\tfrac{2m^2}{c + a} \right)$ & $(a, -b, -c) \mapsto \left( -\tfrac{mb}{c + a}, -\tfrac{2m^2}{c + a} \right)$\Tstrut\Bstrut\\ \cline{1 - 3}
 {\multirow{2}{*}{$\varphi^{-1}(B[n])$}} & $(x,y) \mapsto \left( \tfrac{x^2 - m^2}{y}, \tfrac{2mx}{y}, \tfrac{x^2 + m^2}{y} \right)$ & $(x,-y) \mapsto \left( \tfrac{x^2 - m^2}{y}, -\tfrac{2mx}{y}, \tfrac{x^2 + m^2}{y} \right)$\Tstrut\Bstrut\\ \cline{2 - 3}
 	& $(x,-y) \mapsto \left( -\tfrac{x^2 - m^2}{y}, -\tfrac{2mx}{y}, -\tfrac{x^2 + m^2}{y} \right)$ & $(x,y) \mapsto \left( -\tfrac{x^2 - m^2}{y}, \tfrac{2mx}{y}, -\tfrac{x^2 + m^2}{y} \right)$\Tstrut\Bstrut\\ \cline{1 - 3}
\end{tabular}
\end{center}

The above table shows us that 
\[ \varphi(A[-m]) = \varphi(A[m]) \subseteq B[m] = B[-m] \quad \text{and} \quad \psi(B[-m]) \subseteq A[-m], \] 
and we achieve this with $\varphi$ by replacing $(a, b, c)$ in $A[m]$ with $(a, -b, c)$ 
for $A[-m]$. Finally, one can check that $\psi \circ \varphi = I_{A[-m]}$ and 
$\varphi \circ \psi = I_{B[-m]}$. We summarize this result through the following 
proposition. 

\begin{proposition}\label{sec4prop1}
Let $n$ be a nonzero integer. The system \eqref{eq1} has a solution in 
$\mathbb{Q}$ if and only if the elliptic curve $y^2 = x^3 - n^2 x$ has positive 
rank.
\end{proposition}

We can also refocus Proposition \ref{sec2prop1} into a search for a solution to 
the system \eqref{eq1}. 

\begin{proposition}\label{sec4prop2}
Let $n$ be a nonzero integer. The system \eqref{eq1} has a solution in $\mathbb{Q}$ 
if and only if the same system has a solution in $\mathbb{Q}$ when $n$ is replaced 
$\alpha^2 n$, for some $\alpha \in \mathbb{Z} \setminus \{0\}$.
\end{proposition}

Proposition \ref{sec4prop1} and Proposition \ref{sec4prop2} combined, tells us that 
the elliptic curve $E_1/\mathbb{Q}: y^2 = x^3 - n^2 x$ has positive rank if and 
only if $E_2/\mathbb{Q}: y^2 = x^3 - (\alpha^2 n)^2 x$ has positive rank. This is 
not hard to see as $E_1$ and $E_2$ are in fact isomorphic over $\mathbb{Q}$, and 
the  correspondence is given by 
\begin{align*}
	E_1 &\to E_2\\
	(x,y) &\mapsto (\alpha^2 x, \alpha^3 y)\\
	\left( \tfrac{x}{\alpha^2}, \tfrac{y}{\alpha^3} \right) &\mapsfrom (x,y).
\end{align*}

\begin{remark}\label{sec4rem1}
In general, if $\mu \in \overline{\mathbb{Q}}^{\ast}$, then the elliptic curves 
$E_1/\mathbb{Q}: y^2 = x^3 - Ax$ and $E_2/\mathbb{Q}: y^2 = x^3 - A\mu^4 x$ 
are isomorphic over $\mathbb{Q}$. 
\end{remark}

Recall that Proposition \ref{sec2prop2} says that if $n$ is a nonzero integer for which 
the system \eqref{eq4.1} has a solution, then we can write $nm^2 = uv(u^2 - v^2)$, 
for some integers $u$, $v$, and $m$. We can show that the converse of this result 
holds as well. Suppose we have a nonzero integer $n$ for which the equation 
$nm^2 = uv(u^2 - v^2)$ has a solution $(u, v, m) \in \mathbb{Z}^3$. Dividing both 
sides by $u^4$, we get 
\[ n\left( \frac{m}{u^2} \right)^2 = \left( \frac{u}{v} \right)^3 - \frac{nu}{v}. \]
Then multiply both sides by $n^3$ to obtain
\[ \frac{n^2 m}{u^2} = \frac{nu}{v} - n^2\left( \frac{u}{v} \right). \]
This tells us that the elliptic curve $E_n/\mathbb{Q}: y^2 = x^3 - n^2 x$ has the point 
$\left( \tfrac{n^2 m}{u^2}, \tfrac{nu}{v} \right)$, which is a point of infinite order as 
$\tfrac{nu}{v} \neq 0$. Therefore, $\rank(E_n) \geq 1$. Finally, we refer to the 
following diagram to complete the proof. 
\begin{center}
\begin{tikzpicture}
\node[anchor = north] (A) at (0,0) {\eqref{eq1} has a solution in $\mathbb{Q}$}; 
\node[anchor = north] (B) at (6,0) {\eqref{eq4.1} has a solution in $\mathbb{Z}$}; 
\node[anchor = south] (C) at (6, -2.5) {$nm^2 = uv\left( u^2 - v^2 \right)$};
\node[anchor = south] (D) at (0,-2.5) {$\rank(E_A) \geq 1$}; 
\draw[<->, >=stealth] (A) -- (B);
\node[anchor = south] at (3,-0.25) {\scriptsize{Sec 1}}; 
\draw[->, >=stealth] (C) -- (D);
\draw[->, >=stealth] (B) -- (C);
\node[anchor = west] at (6,-1.25) {\scriptsize{Prop \ref{sec2prop2}}}; 
\draw[<->, >=stealth] (A) -- (D);
\node[anchor = west] at (-1.3,-1.1) {\scriptsize{Prop \ref{sec4prop1}}}; 
\node[anchor = west] at (-1.55,-1.4) {\scriptsize{and Rem \ref{sec4rem1}}}; 
\end{tikzpicture}
\end{center}

The following theorem summarizes the result that we have so far. 

\begin{restatable}{theorem}{mainresult}\label{sec4thm1}
Fix $n \in \mathbb{Z} \setminus\{0\}$. The elliptic curve $y^2 = x^3 - n^2 x$ 
has positive rank over $\mathbb{Q}$ if and only if we can write
\[ nm^2 = uv(u^2 - v^2), \]
for some integers $u$ and $v$.
\end{restatable}

%%%%%%%%%%%%%%%%%%%%%%%%%%%%%%%%%%%%%%%%%%%%%

\section{Proof of Theorem 1 using the method of 2-descent}

The forward direction of Theorem \ref{sec4thm1} is the implication of Proposition 
\ref{sec1prop1} and Proposition \ref{sec2prop1} combined. In the previous section, 
we were able to prove the converse by deriving an elliptic curve from the equation 
$nm^2 = uv\left( u^2 - v^2 \right)$ which we assumed to have a solution in 
$\mathbb{Z}$. In this section, we will show another proof of the converse direction 
using the method of 2-descent. 

We will begin with an integer $n$ that can be written as 
$nm^2 = uv(u^2 - v^2) =: A$, for some integers $u$, $v$, and $m$. Recall that 
the method of 2-descent calculates a lower bound for the rank of the elliptic curve 
$E_n/\mathbb{Q}: y^2 = x^3 - n^2 x$. Now the elliptic curve 
$E_A/\mathbb{Q}: y^2 = x^3 - (nm^2)^2 x = x^3 - Ax$ gives us more to work 
with. But by Remark \ref{sec4rem1}, $E_A$ and $E_n$ are isomorphic over 
$\mathbb{Q}$, so we can work with $E_A$ instead. In Section 3, we saw that 
applying 2-descent to $E_A$ means enumerating equations of the form 
$\mathcal{T}(b_1)$ and $\overline{\mathcal{T}}(\overline{b_1})$ which are 
solvable over $\mathbb{Z}$. Apart from showing that $\rank(E_A) \geq 1$, this 
method will enable us to extract subfamilies of $E_A$ of higher rank later on.\\

During his time, Fermat proved that 1 is not a congruent number. Proposition 
\ref{sec2prop1} extends this fact to squares in $\mathbb{Z}$. We show an analogous 
result: $uv\left( u^2 - v^2 \right)$ can never produce a square or the negative of a 
square in $\mathbb{Z}$. The proofs in this section will make use of the 2-descent 
that we discussed in the previous section. 

\begin{proposition}\label{sec4prop4}
The only solutions to the Diophantine equation $\pm w^2 = uv \left( u^2 - v^2 \right)$ 
are $(u,v,w) = (a,a,0), (0,v,0), (u, 0, 0)$ where $a$ is a nonzero integer.
\end{proposition}
\begin{proof}
Without loss of generality, we only need to show that the Diophantine equation  
$w^2 = uv \left( u^2 - v^2 \right)$ has no nontrivial solutions. Now by a suitable 
change of coordinates, this equation can be transformed into the elliptic curve 
$y^2 = x^3 - x$. An application of Fermat's last theorem (quartic case) tells 
us that this elliptic curve only has $(x,y) = (0, 0), (\pm 1, 0)$ as solutions over 
the rationals. 
\end{proof}

Let $A = nm^2 = uv(u^2 - v^2)$. Suppose $\gcd(u,v) = d$. Then we can write 
$u = du_1$ and $v = dv_1$ where $\gcd(u_1, v_1) = 1$. This results to 
$A = d^4 u_1 v_1 \left( u_1^2 - v_1^2 \right)$. Since our goal is to study the rank 
of $E_A$, by Remark \ref{sec4rem1}, we can ignore the factor $d^4$ in $A$. 
This tells us that we can simply assume that $u$ and $v$ are relatively prime in 
the first place. 

Consider the case when $uv$ is a square or the negative of a square. With the 
recent assumption we made above, this can only mean that both $|u|$ and $|v|$ 
are squares. Suppose we have $|u| = u_1^2$ and $|v| = v_1^2$ with 
$\gcd(u_1, v_1) = 1$. So $A = \pm u_1^2 v_1^2 \left( u_1^4 - v_1^4 \right)$. 
Again, by Remark \ref{sec4rem1}, we can leave out the squares and instead 
consider $E_A: y^2 = x^3 - A^2 x$ with $A = \left( u_1^4 - v_1^4 \right)$. Below 
we have a lemma that will help us determine the rank of $E_A$.

\begin{lemma}\label{sec4lem1}
The equation
	\begin{equation}
		x^2 + y^2 = 2z^2 \label{sec4lem1eq}
	\end{equation}
has no integer solutions when $x$ and $y$ have opposite parities. When $x$ 
and $y$ are relatively prime odd integers, the complete solution is given by 
$(x,y,z) = (r^2 + 2rs - s^2, r^2 - 2rs - s^2, r^2 + s^2)$ for relatively prime 
integers $r$ and $s$.
\end{lemma}
\begin{proof}
Since the right hand side of equation \eqref{sec4lem1eq} is even, $x$ and $y$ can 
only be both even or both odd. This is precisely the contrapositive of the first 
statement of the lemma. The second statement can be obtained by writing 
$x= a + b$ and $y = a-b$ for some $a, b \in \mathbb{Z}$ and reducing equation 
\eqref{sec4lem1eq} to the Pythagorean equation $a^2 + b^2 = z^2$. The solution is then 
obtained from the well-known parametrization of the Pythagorean triples.  
\end{proof}

\begin{proposition}\label{sec5prop1}
Let $u$ and $v$ be relatively prime nonzero integers. The elliptic curve 
$E_A: y^2 = x^3 - A^2 x$ with $A= u^4 - v^4$, has rank 
at least 1.
\end{proposition}
\begin{proof} 
Fix $A = u^4 - v^4$. Because of Remark \ref{sec4rem1}, we may assume that 
$u^4 - v^4$ cannot be written as $\alpha^2\left( u_1^4 - v_1^4 \right)$, for some 
integers $u_1$, $v_1$, and $\alpha$ with 
$\left| u_1^4 - v_1^4 \right| < \left| u^4 - v^4 \right|$.
Consider the isogenous elliptic curves 
\[
	E_A: y^2 = x \left[ x^2 - \left( u^4 - v^4 \right)^2 \right] 
	\quad \text{and} \quad 
	\overline{E_A}: y^2 = x \left[ x^2 + 4\left( u^4 - v^4 \right)^2 \right].
\]
We enumerate some solvable quartic equations $\mathcal{T}(b_1)$, 
$\overline{\mathcal{T}}(\overline{b_1})$ as discussed in Section 2 to determine the 
orders of 
$\alpha(E_A(\mathbb{Q}))$ and $\overline{\alpha}(\overline{E_A}(\mathbb{Q}))$.
\begin{table}[H]
\renewcommand{\arraystretch}{1.3}
\begin{tabular}{ c | c | c }
 & Quartic equation & Solution $(N, e, M)$\\ \hline
$\mathcal{T}(1)$ & $N^2 = M^4 - \left( u^4 - v^4 \right)^2 e^4$ & $(1, 0, 1)$\\
$\mathcal{T}(-1)$ & $N^2 = - \left[ M^4 - \left( u^4 - v^4 \right)^2 e^4 \right]$ & $(A, 1, 0)$\\
$\mathcal{T}(\pm A)$ & $N^2 = \pm \left( u^4 - v^4 \right) \left[ M^4 - e^4 \right]$ & $(0, 1, 1)$\\
$\mathcal{\overline{T}}(1)$ & $N^2 = M^4 + 4 \left( u^4 - v^4 \right)^2 e^4$ & $\left( u^8 - v^8, uv, u^4 - v^4 \right)$\\
$\mathcal{\overline{T}}\left( 2\left( u^2 + v^2 \right) \right)$ & $N^2 = 2 \left( u^2 + v^2 \right) \left[ M^4 + \left( u^2 - v^2 \right)^2 e^4 \right]$ & $\left( 2(u - v) \left( u^2 + v^2 \right), 1, u - v \right)$
\end{tabular}
%\caption{}
\end{table}

\noindent By Proposition \ref{eltstopts} we realize that the elements $-1, +1, -A, +A$ 
of $\alpha(E_A(\mathbb{Q}))$ give rise to the finite ordered points in 
$E_A(\mathbb{Q})$, namely $\mathcal{O}$, $(0,0)$, $(-A,0)$, and $(A,0)$, 
respectively. Moreover, by Proposition \ref{sec4prop4}, $A$ cannot be a perfect 
square. So $\pm 1$, $\pm A$ are distinct elements in 
$\alpha\left( E_A(\mathbb{Q}) \right)$. Now if 1 and 
$2 \left( u^2 + v^2 \right)$ are distinct elements in 
$\overline{\alpha}(\overline{E_A}(\mathbb{Q}))$, then we get 
$\#\alpha(E_A(\mathbb{Q})) \geq 4$ and 
$\#\overline{\alpha}(\overline{E_A}(\mathbb{Q})) \geq 2$. 

The problem is this, is not always the case. The elements 1 and 
$2\left( u^2 + v^2 \right)$ are equal in 
$\overline{\alpha}(\overline{E_A}(\mathbb{Q}))$ if there nonzero integer 
$\beta$ that satisfies $u^2 + v^2 = 2\beta^2$. By Lemma \ref{sec4lem1} 
such a $\beta$ cannot exist when $u$ and $v$ have opposite parities. Now we 
can exclude the case when both $u$ and $v$ are even because they are 
assumed to be relatively prime. This leaves the case when $u$ and $v$ are 
both odd. 
%(Note there are still some choices of odd integers $u$ and $v$ 
%\emph{e.g.}, $(u,v) = (3,1)$ where all the enumerated torsors 
%are distinct.) 
Suppose $u$ and $v$ are odd and that there is a nonzero integer $\beta$ so 
that $u^2 + v^2 = 2\beta^2$. The second statement of Lemma \ref{sec4lem1} 
tells us that we can write $u = r^2 + 2rs - s^2$, $v = r^2 - 2rs - s^2$, and 
$\beta = r^2 + s^2$ for some relatively prime integers $r$ and $s$. Moreover, 
$r$ and $s$ necessarily have opposite parities otherwise $2 \mid \gcd(u,v)$. 
This gives us
\begin{equation}\label{form1}
A =  u^4 - v^4 = 16rs \left( r^2 - s^2 \right)\beta^2, \quad \beta^2 \neq 1.
\end{equation}
We consider the following quartic equations in place of 
$\overline{\mathcal{T}}\left( 2\left( u^2 + v^2 \right) \right)$:

\begin{table}[H]
\renewcommand{\arraystretch}{1.3}
\begin{tabular}{ c | c | c }
 & Quartic equation & Solution $(N, e, M)$\\ \hline
$\mathcal{T}\left( 16\left( r^2 - s^2 \right) \beta^2 \right)$ & $N^2 = 16\left( r^2 - s^2 \right) \beta^2 \left[ M^4 - r^2 s^2 e^4 \right]$ & $\left( 4r\beta\left( r^2 - s^2 \right), 1, r \right)$ \Tstrut\\
$\mathcal{T}\left( -16\left( r^2 - s^2 \right) \beta^2 \right)$ & $N^2 = -16\left( r^2 - s^2 \right) \beta^2 \left[ M^4 - r^2 s^2 e^4 \right]$ & $\left( 4s\beta\left( r^2 - s^2 \right), 1, s \right)$\\
\end{tabular}
\end{table}

\begin{enumerate}[label = (\Roman*)]
	\item{If $r^2 - s^2 \not\equiv \pm 1 \left( \bmod\ (\mathbb{Q}^{\times})^2 \right)$, 
then $\pm 16\left( r^2 - s^2 \right)\beta^2$ and $\pm 1$ are distinct elements 
in $\alpha(E_n(\mathbb{Q}))$. We only need to show that 
$\pm 16\left( r^2 - s^2 \right)\beta^2$, $\pm A$ are distinct elements in 
$\alpha(E_A(\mathbb{Q}))$. Suppose otherwise, then we get
$rs \equiv \pm 1 \left( \bmod\ (\mathbb{Q}^{\times})^2 \right)$. Since $r$ 
and $s$ are relatively prime, we can write $|r| = r_1^2$ and $|s| = s_1^2$ for 
some nonzero integers $r_1$ and $s_1$. Since $\beta \neq 1$, this results to:
$$ A = u^4 - v^4 = \pm 16r_1^2 s_1^2 \left( r_1^4 - s_1^4 \right) \beta^2, \text{ with } \left| r_1^4 -s_1^4 \right| < \left| u^4 - v^4 \right|. $$
By an earlier assumption, this cannot be the case. As $\alpha(E_A(\mathbb{Q}))$ 
is a multiplicative group, we have $\pm rs \in \alpha(E_A(\mathbb{Q}))$. Hence in 
this case, $\#\alpha(E_A(\mathbb{Q})) \geq 8$ and 
$\#\overline{\alpha}(\overline{E_A}(\mathbb{Q})) \geq 1$. 

\begin{table}[H]
\begin{tabular}{ c | c | c }
 & Quartic equation & Solution $(N, e, M)$\\ \hline
$\mathcal{T} \left(\pm 16rs\beta^2 \right)$ & $N^2 = \pm 16rs \left[ M^4 - \left( r^2 - s^2 \right)^2 e^4 \right]$ & $(8rs(r \pm s), 1, r \pm s)$ \Tstrut\\
\end{tabular}
%\caption{}
\end{table}}

	\item Now if $r^2 - s^2 \equiv 1 \left( \bmod\ (\mathbb{Q}^{\times})^2 \right)$, 
then we write $ r^2 - s^2 = \gamma^2 $ for some nonzero integer $\gamma$. 
This is the familiar Pythagorean equation which we can parametrize by either
\[\begin{cases}
	\begin{aligned}
		r &= w^2 + z^2\\
		s &= 2wz\\ 
		\gamma &= w^2 - z^2
	\end{aligned}
\end{cases}
\quad \text{or} \quad
\begin{cases}
	\begin{aligned}
		r &= w^2 + z^2\\
		s &= w^2 - z^2\\
		\gamma &= 2wz.
	\end{aligned}
\end{cases}\]
Here, we take $w$ and $z$ to be nonzero relatively prime integers with opposite 
parities, otherwise we will get $2 \mid \gcd(r,s)$, a contradiction. If instead we 
have $r^2 - s^2 \equiv -1 \left( \bmod\ (\mathbb{Q}^{\times})^2 \right)$, 
notice that we can rewrite this as 
$s^2 - r^2 \equiv 1 \left( \bmod\ (\mathbb{Q}^{\times})^2 \right)$. So we do the 
same but the roles of $r$ and $s$ are switched. Going back, note that we cannot use 
the second parametrization as it would imply that $r$ and $s$ are both odd, contrary 
to the assumption that $r$ and $s$ have opposite parities. Instead, we apply the first 
parametrization to obtain
\begin{equation}\label{form2}
	A = 32wz \left( w^2 + z^2 \right)\beta^2\gamma^2, \quad \beta^2, \gamma^2 \neq 1.
\end{equation}
For this $A$, we have $64 \left( w^2 + z^2 \right)\beta^2\gamma^2
\in \overline{\alpha}(\overline{E_A}(\mathbb{Q}))$. 

\begin{table}[H]
\begin{tabular}{ c | c | c }
 & Quartic equation & Solution $(N, e, M)$\\ \hline
$\overline{\mathcal{T}} \left( 64 \left( w^2 + z^2 \right) \beta^2 \gamma^2 \right)$ & $N^2 = 64 \left( w^2 + z^2 \right) \beta^2\gamma^2 \left[ M^4 + w^2 z^2 e^4 \right]$ & $\left( 8w\beta\gamma\left( w^2 + z^2 \right), 1, w \right)$ \Tstrut\\
\end{tabular}
%\caption{Some solvable torsors under the homomorphisms $\alpha$ and $\overline{\alpha}$}
\end{table}

\begin{enumerate}[label = (\alph*)]
	\item{If $w^2 + z^2 \not\equiv 1 \left( \bmod\ (\mathbb{Q}^{\times})^2 \right)$
then $\#\alpha(E_A(\mathbb{Q})) \geq 4$ 
and $\#\overline{\alpha}(\overline{E_A}(\mathbb{Q})) \geq 2$.} 

	\item{If $w^2 + z^2 \equiv 1 \left( \bmod\ (\mathbb{Q}^{\times})^2 \right)$, 
then $w^2 + z^2 = \varepsilon^2$, for some nonzero integer $\varepsilon$. 
This has a well known parametrization
\[\begin{cases}
	\begin{aligned}
		w &= a^2 - b^2\\
		z &= 2ab\\ 
		\delta &= a^2 + b^2
	\end{aligned}
\end{cases}
\quad \text{or} \quad
\begin{cases}
	\begin{aligned}
		w &= 2ab\\
		z &= a^2 - b^2\\
		\delta &= a^2 + b^2.
	\end{aligned}
\end{cases}\]
In any choice of parametrization, we obtain 
\[ A = 64ab \left( a^2 - b^2 \right) \beta^2 \gamma^2 \delta^2 , \quad \beta^2, \gamma^2, \delta^2 \neq 1. \]

This brings us back to the familiar \eqref{form1}. Notice that from here on, the 
process is repetitive. }
\end{enumerate}
\end{enumerate}
We summarize this process below. Note that we there are four sure elements 
of $\alpha(E_A(\mathbb{Q}))$: $\pm 1$, and $\pm A$. Moreover, we know for 
certain that $1 \in \overline{\alpha}(\overline{E_A}(\mathbb{Q}))$. When $A$ 
is of the form $ab\left( a^2 - b^2 \right)\square$, we get four new elements of 
$\#\alpha(E_A(\mathbb{Q}))$. Its difference 
from the existing elements of the group depends on a congruence of the form 
$X^2 - Y^2 \equiv 1 \left( \bmod\ \mathbb{Q}^{\times} \right)$. 
If this is not satisfied then the procedure terminates with 
\[ \#\alpha(E_A(\mathbb{Q})) \geq 8
\quad \text{and} \quad
\#\overline{\alpha}(\overline{E_A}(\mathbb{Q})) \geq 1. \]
If it is, then we get the familiar \eqref{form2}, that is $n$ is of the form 
$2ab(a^2 + b^2)\square$. In here, $\#\alpha(E(\mathbb{Q})) \geq 4$ and 
we can increase $\#\overline{\alpha}(\overline{E}(\mathbb{Q}))$ to at least 
2 by considering a congruence of the form 
$X^2 + Y^2 \equiv 1 \left( \bmod\ \mathbb{Q}^{\times} \right)$. 
If this is not satisfied then the process terminates with
\[ \#\alpha(E_A(\mathbb{Q})) \geq 4
\quad \text{and} \quad
\#\overline{\alpha}(\overline{E_A}(\mathbb{Q})) \geq 2. \]
If it is, then we go back to \eqref{form1}. The important thing to notice is that 
this procedure brings out the nontrivial square factors of $A$. Since $A$ is an 
integer with only finitely many factors, this process is sure to conclude at some 
point.

In any case, by the rank formula \eqref{rankformula} in Proposition 
\ref{sec3prop1}, we get
\[ 2\strut^{\rank(E_A)} \geq \dfrac{4 \cdot 2}{4} \text{\ \ or\ \ } \dfrac{8 \cdot 1}{4}. \]
Hence we have $\rank(E_A) \geq 1$. This completes the proof the proposition.
\end{proof}

%Now we tackle the case when $\gcd(u,v) > 1$. If $d = \gcd(u,v) > 1$, then we 
%can write $u = du_1$ and $v = dv_1$ with $(u_1, v_1) = 1$. This gives us 
%$n = d^4 u_1 v_1 \left( u_1^2 - v_1^2 \right)$. This tells us that with 
%Proposition \ref{sec2prop1}, we can assume that $u$ and $v$ are relatively 
%prime in the first place. The main difference with the previous case is that this 
%time, $u$ and $v$ are not squares in $\mathbb{Z}$.

To complete the proof of the convers of Theorem \ref{sec4thm1}, we discuss  
the general case. 

\begin{proposition}\label{sec5prop2}
Let $u$ and $v$ be relatively prime nonzero integers. The elliptic curve 
$E_A: y^2 = x^3 - A^2 x$ with $A= uv(u^2 - v^2)$ has rank at least one. 
\end{proposition}
\begin{proof}
Consider the isogenous elliptic curves 
\[
	E_A: y^2 = x \left[ x^2 - u^2 v^2 \left( u^2 - v^2 \right)^2 \right] 
	\quad \text{and} \quad
	\overline{E_A}: y^2 = x \left[ x^2 + 4 u^2 v^2 \left( u^2 - v^2 \right)^2 \right].
\]
Similar to Proposition \ref{sec5prop1}, we enumerate some solvable quartic 
equations $\mathcal{T}(b_1)$, $\overline{\mathcal{T}}(\overline{b_1})$ 
to determine the orders of $\alpha(E_A(\mathbb{Q}))$ and 
$\overline{\alpha}(\overline{E_A}(\mathbb{Q}))$.
\begin{table}[H]
\renewcommand{\arraystretch}{1.3}
\begin{tabular}{ c | c | c }
 & Quartic equation & Solution $(N, e, M)$\\ \hline
$\mathcal{T}(1)$ & $N^2 = \left[ M^4 - u^2 v^2 \left( u^2 - v^2 \right)^2 e^4 \right]$ & $(1, 0, 1)$\\
$\mathcal{T}(-1)$ & $N^2 = -\left[ M^4 - u^2 v^2 \left( u^2 - v^2 \right)^2 e^4 \right]$ & $(A, 1, 0)$\\
$\mathcal{T}(\pm A)$ & $N^2 = \pm uv \left( u^2 - v^2 \right) \left[ M^4 - e^4 \right]$ & $(0, 1, 1)$\\
$\mathcal{T}(\pm uv)$ & $N^2 = \pm uv \left[ M^4 - \left( u^2 - v^2 \right)^2 e^4 \right]$ & $\left( 2uv(u \pm v), 1, u \pm v \right)$\\
$\mathcal{T}\left( u^2 - v^2 \right)$ & $N^2 = \left( u^2 - v^2 \right) \left[ M^4 - u^2 v^2 e^4 \right]$ & $\left( u\left( u^2 - v^2 \right), 1, u \right)$\\
$\mathcal{T}\left( -\left( u^2 - v^2 \right) \right)$ & $N^2 = -\left( u^2 - v^2 \right) \left[ M^4 - u^2 v^2 e^4 \right]$ & $\left( v\left( u^2 - v^2 \right), 1, v \right)$\\
$\mathcal{\overline{T}}(1)$ & $N^2 = M^4 + 4 u^2 v^2 \left( u^2 - v^2 \right)^2 e^4$ & $\left( u^4 - v^4, 1, u^2 - v^2 \right)$\\
\end{tabular}
%\caption{Solvable quartic equations for $uv\left( u^2 - v^2 \right)$}
\end{table}

\noindent As in Proposition \ref{sec5prop1}, $\pm 1$, $\pm A$ are the four distinct 
elements of $\alpha(E_n(\mathbb{Q}))$ which correspond to the points in 
$\tor{E_A/(\mathbb{Q})}$. If in particular, $\pm 1$, $\pm uv$, 
$\pm \left( u^2 - v^2 \right)$, and $\pm A$ are distinct as elements of 
$\alpha(E_A(\mathbb{Q}))$, then $\#\alpha(E_A(\mathbb{Q})) \geq 8$ 
and $\#\overline{\alpha}(\overline{E_A}(\mathbb{Q})) \geq 1$. 

Notice that:
\begin{align*}
	 uv \equiv \pm 1 \left( \bmod\ (\mathbb{Q}^{\times})^2 \right) 
	&\iff 
	\left( u^2 - v^2 \right) \equiv \pm A \left( \bmod\ (\mathbb{Q}^{\times})^2 \right)\\
	\left( u^2 - v^2 \right) \equiv \pm 1 \left( \bmod\ (\mathbb{Q}^{\times})^2 \right) 
	&\iff 
	uv \equiv \pm A \left( \bmod\ (\mathbb{Q}^{\times})^2 \right).
\end{align*}
Thus it is enough to look at the case when $\pm uv$ or 
$\pm \left( u^2 - v^2 \right)$ coincide with $\pm 1$ in 
$\alpha(E_A(\mathbb{Q}))$. The case when 
$uv \equiv \pm 1 \left( \bmod\ (\mathbb{Q}^{\times})^2 \right)$ 
is precisely Proposition \ref{sec5prop1}. The case when 
$ \left( u^2 - v^2 \right) \equiv 1 \left( \bmod\ (\mathbb{Q}^{\times})^2 \right)$, 
can also be applied with Proposition \ref{sec5prop1}. In particular, we can 
let $u_1 = u + v$ and $v_1 = u - v$ and consider 
\[ A_1 = u_1 v_1 \left( u_1^2 - v_1^2 \right). \]
Notice that $A_1 = 4A$. And so in this case, we apply Proposition \ref{sec5prop1} 
to $A_1$ where $u_1v_1 \equiv \pm 1 \left( \bmod\ (\mathbb{Q}^{\times})^2 \right)$.
Finally, we note that in any case, we get 
\[ 2\strut^{\rank(E_A)} \geq \dfrac{4 \cdot 2}{4} \text{ or } \dfrac{8 \cdot 1}{4}. \]
Hence, $\rank(E_A) \geq 1$.
\end{proof}

%%%%%%%%%%%%%%%%%%%%%%%%%%%%%%%%%%%%%%%%%%%%%

\section{Constructing CN elliptic curves with rank more than 1} 

The method of 2-descent reduces the computation of the rank of some elliptic 
curves to solving quartic Diophantine equations of the form 
$\mathcal{T}(b_1)$ and $\overline{\mathcal{T}}(\overline{b_1})$. Let 
$A = uv \left( u^2 - v^2 \right)$. From Theorem \ref{sec4thm1} we already 
know that the family of elliptic curves $E_A: y^2 = x^3 - A^2 x$ has rank at least 
1. But it was from the proof of Proposition \ref{sec5prop1} and Proposition 
\ref{sec5prop2} that we obtained a list of the quartic equations which are solvable. 
To get a subfamily of $E_A$ with higher rank, we try to solve quartic equations 
not on our list by imposing conditions on $u$ and $v$, whenever needed. This gives 
us elements of $\alpha(E_A(\mathbb{Q}))$ and 
$\overline{\alpha}(\overline{E_A}(\mathbb{Q}))$. The last step is to count the 
distinct elements of these multiplicative groups and then use equation 
\eqref{rankformula} to compute for the rank. 

%%%%%%%%%%%%%%%%%%%%%%%%%%%%%%%%%%%%%%%%%%%%%

\subsection{A subfamily from parametrization of conics}

Consider the quartic equation
\begin{equation}\label{subsec5.1}
	N^2 = v(u + v) \left[ M^4 - u^2 \left( u - v \right)^2 e^4 \right].
\end{equation}
When $e = 1$ and $M = u$, we get 
$N^2 = u^2 v^2 \left( 2u^2 + uv - v^2 \right)$. Let us choose $u$ and $v$ 
$w^2 = 2u^2 + uv - v^2$ has a solution in $\mathbb{Z}$. Equivalently, we 
can choose $t = \tfrac{u}{v} \in \mathbb{Q}$ is a solution to the quadratic 
equation $w_1^2 = 2t^2 + t - 1$. The existence of a nontrivial rational solution 
$(2, \pm 3)$ tells us that there are actually infinitely many rational numbers 
which satisfy this equation. Moreover, we can obtain a parametrization for the 
solution $t$:
\[ u = r^2 + s^2 \quad \text{and} \quad v = 2r^2 - s^2. \]
Here, we choose $r$ and $s$ to be relatively prime integers. This makes 
\[ A = -3r^2 \left( r^2 + s^2 \right)\left( r^2 - 2s^2 \right)\left( 2r^2 - s^2 \right) \]
This adds $\pm v(u + v)$ to $\alpha(E_A(\mathbb{Q}))$ with the 
corresponding quadruple to \eqref{subsec5.1}
\[(N, e, M) = \left( uvw, 1, u \right) = \left( 3rs\left( r^2 + s^2 \right) \left( 2r^2 - s^2 \right), 1, r^2 + s^2 \right).\] 
Because $\alpha(E_A(\mathbb{Q}))$ is a multiplicative group, $\pm v(u - v)$, 
$\pm u(u + v)$, and $\pm u(u - v)$ are also elements of 
$\alpha(E_A(\mathbb{Q}))$. We claim that we have a subfamily 
\[E_A^1\,/\,\mathbb{Q}: y^2 = x^3 - n^2 x, \text{ with }
	A = -3r^2 \left( r^2 + s^2 \right)\left( r^2 - 2s^2 \right)\left( 2r^2 - s^2 \right) \] 
of $E_A$ with rank at least 2. To do this, we prove that the 16 elements we 
have identified in $\alpha(E_n^1(\mathbb{Q}))$ are distinct. 

\begin{lemma}\label{lem1const2}
Let $u$ and $v$ be as above. For all but finitely many choices of relatively 
prime integers $r$ and $s$, $\pm 1, \pm A, \pm uv, \pm\left( u^2 - v^2 \right)$ 
define eight distinct elements in $\alpha(E_A^1(\mathbb{Q}))$.
\end{lemma}
\begin{proof}
We already know from Section 4 that $\pm 1$, $\pm A$ are distinct elements 
of $\alpha(E_A^1(\mathbb{Q}))$ corresponding to the 2-torsion points of $E_A^1$. 
First we note that it cannot happen that 
$uv \equiv \pm u^2 - v^2 \left( \bmod\ (\mathbb{Q}^{\times})^2 \right)$, 
otherwise $A$ would be a square or the negative of a square, which is impossible 
by Proposition \ref{sec4prop4}.
Moreover, we have seen that 
\begin{align*}
	 uv \equiv \pm 1 \left( \bmod\ (\mathbb{Q}^{\times})^2 \right) 
	&\iff 
	u^2 - v^2 \equiv \pm A \left( \bmod\ (\mathbb{Q}^{\times})^2 \right)\\
	u^2 - v^2 \equiv \pm 1 \left( \bmod\ (\mathbb{Q}^{\times})^2 \right) 
	&\iff 
	uv \equiv \pm A \left( \bmod\ (\mathbb{Q}^{\times})^2 \right).
\end{align*}
So it is enough to show that  
$uv \equiv \pm 1 \left( \bmod\ (\mathbb{Q}^{\times})^2 \right)$ or 
$u^2 - v^2 \equiv \pm 1 \left( \bmod\ (\mathbb{Q}^{\times})^2 \right)$ 
are true for only finitely many relatively prime integers $r$ and $s$. We discuss 
the latter first. Suppose the congruence holds. In terms of $r$ and $s$, we have 
$3r^2 \left( 2s^2 - r^2 \right) \equiv \pm 1 \left( \bmod\ (\mathbb{Q}^{\times})^2 \right)$, 
which implies that $3\left( 2s^2 - r^2 \right) = \pm y^2$, for some nonzero 
$y \in \mathbb{Q}$. This tells us that $3 \mid \pm \left( 2s^2 - r^2 \right)$. 
However, by plugging in all possible values modulo 3, one can see that the 
congruence $2s^2 - r^2 \equiv 0 \pmod 3$ only holds if $3 \mid \gcd(r,s)$. 
This is a contradiction.

Finally we discuss the former case. In terms of $r$ and $s$, we get the 
Diophantine equation $\pm y^2 = - s^4 + r^2 s^2 + 2r^4$. When the 
left hand side is positive, we can reduce the equation modulo 4 to get a 
contradiction. All that is left is to check the integer solutions of 
$y^2 = s^4 - r^2 s^2 + 2r^4$. We move the discussion to $\mathbb{Q}$ 
and study the quartic curve $y^2 = 2x^4 - x^2 + 1$ which has a rational 
point $(0, \pm 1)$. This tells us that this curve is also birationally equivalent 
to an elliptic curve (in Weierstrass form). We can compute that the Jacobian 
of this elliptic curve has rank 0. So there can only be finitely many 
rational pairs $(x,y)$ on the quartic and hence finitely many pairs $r$ and $s$ 
satisfying the congruence 
$uv \equiv -1 \left( \bmod\ (\mathbb{Q}^{\times})^2 \right)$. This completes 
the proof of the lemma. 
\end{proof}

\begin{lemma}\label{lem2const2}
Let $\mathcal{S} := \{ \pm u(u + v), \pm u(u - v), \pm v(u + v), \pm v(u - v) \}$, 
where $u$ and $v$ are parametrized in terms of $r$ and $s$ as above. Then for all 
but finitely many relatively prime integers $m$ and $n$, $\mathcal{S}$ 
defines eight distinct elements in $\alpha(E_A^1(\mathbb{Q}))$.
\end{lemma}
\begin{proof}
We proceed by cases. Each of the following congruences below 
\[\begin{cases}
	\begin{aligned}
		u(u + v) &\equiv \pm v(u - v) \left( \bmod\ (\mathbb{Q}^{\times})^2 \right)\\
		v(u + v) &\equiv \pm u(u - v) \left( \bmod\ (\mathbb{Q}^{\times})^2 \right)
	\end{aligned}
\end{cases}\]
suggests that $n$ is a square or the negative of a square which cannot be the 
case according to Proposition \ref{sec4prop4}. This is a contradiction. The next 
four congruences 
\[\begin{cases}
	\begin{aligned}
		u(u + v) &\equiv \pm v(u + v) \left( \bmod\ (\mathbb{Q}^{\times})^2 \right)\\
		u(u - v) &\equiv \pm v(u - v) \left( \bmod\ (\mathbb{Q}^{\times})^2 \right)
	\end{aligned}
\end{cases}\]
imply that $uv \equiv \pm 1 \left( \bmod\ (\mathbb{Q}^{\times})^2 \right)$.  
We already know from Lemma \ref{lem1const2} that this is possible only for 
finitely many choices of relatively prime integers $r$ and $s$. The last four 
congruences 
\[\begin{cases}
	\begin{aligned}
		u(u + v) &\equiv \pm u(u - v) \left( \bmod\ (\mathbb{Q}^{\times})^2 \right)\\
		v(u + v) &\equiv \pm v(u - v) \left( \bmod\ (\mathbb{Q}^{\times})^2 \right)
	\end{aligned}
\end{cases}\]
give rise to $\pm 3 \equiv 2r^2 - s^2$. Modulo 4, we see that there is no pair of 
relatively prime integers $(r,s)$ that can satisfy this congruence. This is a contradiction. 
Having discussed all possible cases, we have completed the proof of the lemma. 
\end{proof}

\begin{lemma}\label{lem3const2}
Let $\mathcal{S}$ be as in Lemma \ref{lem2const2}. For all but finitely many 
choices of relatively prime integers $r$ and $s$, each element of $\mathcal{S}$ 
is distinct from $\pm 1$ or $\pm A$ in $\alpha(E_A^1(\mathbb{Q}))$.
\end{lemma}
\begin{proof}
Notice that if $\alpha \in \mathcal{S}$ is equal to $\pm 1$ in 
$\alpha(E_A^1(\mathbb{Q}))$, then $\tfrac{A}{\alpha} \in \mathcal{S}$ is equal 
to $\pm A$ in $\alpha(E_A^1(\mathbb{Q}))$. So it is enough to show that each 
element of $\mathcal{S}$ is distinct from $\pm 1$ to prove the statement. Now 
observe that from each congruence
\[\begin{cases}
	\begin{aligned}
		u(u - v) \equiv \pm 1 \left( \bmod\ (\mathbb{Q}^{\times})^2 \right)\\
		v(u - v) \equiv \pm 1 \left( \bmod\ (\mathbb{Q}^{\times})^2 \right)
	\end{aligned}
\end{cases}\]
arises a Diophantine equation 
\[\begin{cases}
	\begin{aligned}
		\pm y^2 &= -r^4 + r^2 s^2 + 2s^4\\
		\pm y^2 &= -2r^4 + 5r^2 s^2 - 2s^4,
	\end{aligned}
\end{cases}\]
respectively. We have already dealt with equations similar to the first two in 
Lemma \ref{lem1const2} ($r$ and $s$ are switched). These will also lead to a 
contradiction. We direct our attention to the other two: 
$\pm y^2 = -2r^4 + 5r^2 s^2 - 2s^4$. When the left hand side 
is negative, we can reduce the equation modulo 4 to get a contradiction. 
All that is left is to check the integer solutions of 
$y^2 = -2r^4 + 5r^2 s^2 - 2s^4$. Using the same technique as before, we 
move the setting to $\mathbb{Q}$ and study the quartic curve 
$y^2 = -2x^4 + 5x^2 - 2$. This curve has a rational point $(1, \pm 1)$, which 
tells us that it is birationally equivalent to an elliptic curve (in Weierstrass form). 
We can compute that the Jacobian of this elliptic curve has rank 0. Hence, there 
can only be finitely many rational pairs $(x,y)$ on the quartic and consequently 
only finitely many pairs $(r,s)$ of relatively prime integers satisfying the 
congruence $v(u - v) \equiv 1 \left( \bmod\ (\mathbb{Q}^{\times})^2 \right)$. 

Finally, from each congruence
\[\begin{cases}
	\begin{aligned}
		u(u + v) \equiv \pm 1 \left( \bmod\ (\mathbb{Q}^{\times})^2 \right)\\
		v(u + v) \equiv \pm 1 \left( \bmod\ (\mathbb{Q}^{\times})^2 \right)
	\end{aligned}
\end{cases}\]
arises a Diophantine equation
\[\begin{cases}
	\begin{aligned}
		\pm y^2 &= 3r^2 + 3s^2\\
		\pm y^2 &= 6r^2 - 3s^2,
	\end{aligned}
\end{cases}\]
respectively. We have already dealt with equations similar to the last two in 
Lemma \ref{lem1const2} (switching $r$ and $s$). These will also lead to a 
contradiction. Moving on to $\pm y^2 = 3r^2 + 3s^2$, it is clear that when 
the left hand side is negative, there are no solutions to the equation. On the 
other hand, when the left hand side is positive, one can reduce the equations 
modulo 4 to obtain a contradiction. We have discussed all possible cases so 
this ends the proof for the lemma.
\end{proof}

\begin{lemma}\label{lem4const2}
Let $\mathcal{S}$ be as in Lemma \ref{lem2const2}. For all but finitely many 
choices of relatively prime integers $r$ and $s$, each element of $\mathcal{S}$ 
is distinct from $\pm uv$ or $\pm u^2 - v^2$ in $\alpha(E_A^1(\mathbb{Q}))$.
\end{lemma}
\begin{proof}
Similar to Lemma \ref{lem3const2}, if $\alpha \in \mathcal{S}$ is equal to 
$\pm uv$ in $\alpha(E_A^1(\mathbb{Q}))$, then $\tfrac{A}{\alpha} \in \mathcal{S}$ 
is equal to $\pm \left( u^2 - v^2 \right)$ in $\alpha(E_n^1(\mathbb{Q}))$. So it 
is enough to claim that any element of $\mathcal{S}$ is distinct from $\pm uv$ 
to prove the statement. Notice the following equivalence:
\[\begin{cases}
	\begin{aligned}
		u(u + v) \equiv \pm uv \left( \bmod\ (\mathbb{Q}^{\times})^2 \right) &\iff v(u + v) \equiv \pm 1 \left( \bmod\ (\mathbb{Q}^{\times})^2 \right)\\
		u(u - v) \equiv \pm uv \left( \bmod\ (\mathbb{Q}^{\times})^2 \right) &\iff v(u - v) \equiv \pm 1 \left( \bmod\ (\mathbb{Q}^{\times})^2 \right)\\
		v(u + v) \equiv \pm uv \left( \bmod\ (\mathbb{Q}^{\times})^2 \right) &\iff u(u + v) \equiv \pm 1 \left( \bmod\ (\mathbb{Q}^{\times})^2 \right)\\
		v(u - v) \equiv \pm uv \left( \bmod\ (\mathbb{Q}^{\times})^2 \right) &\iff u(u - v) \equiv \pm 1 \left( \bmod\ (\mathbb{Q}^{\times})^2 \right).
	\end{aligned}
\end{cases}\]
We already know from Lemma \ref{lem2const2} that only finitely many choices 
of relatively prime integers $r$ and $s$ can satisfy any of these congruences. 
\end{proof}
These previous lemmas show that $\#\alpha(E_A^1(\mathbb{Q})) \geq 16$. So we 
obtain a subfamily 
\[E_{A}^1\,/\,\mathbb{Q}: y^2 = x^3 - n^2 x, \text{ with }
	A = -3r^2 \left( r^2 + s^2 \right)\left( r^2 - 2s^2 \right)\left( 2r^2 - s^2 \right) \] 
of $E_A$ with rank at least 2. The following is a table which shows the rank of 
$E_A^1$ for various pairs of relatively prime integers $(r, s)$. 

\begin{table}[H]
\renewcommand{\arraystretch}{1.1}
\begin{tabular}{| c | c | c |}
\hline
$(r,s)$ & $A$ & $\rank(E_A^1)$\\ \hline
$(1,2)$ & $-210$ & 2\\ \hline
$(1,3)$ & $-3570$ & 2\\ \hline
$(3,4)$ & 31050 & 2\\ \hline
$(1,4)$ & $-22134$ & 3\\ \hline
$(5,6)$ & 3010350 & 3\\ \hline
$(4,7)$ & $-4349280$ & 3\\ \hline
$(5,7)$ & 405150 & 3\\ \hline
$(6,7)$ & 13090680 & 3\\ \hline
$(1,8)$ & $-1535430$ & 3\\ \hline
$(4,9)$ & $-33309024$ & 3\\ \hline
$(4,11)$ & $-132269664$ & 3\\ \hline
$(3,5)$ & $-263466$ & 4\\ \hline
$(4,5)$ & 468384 & 4\\ \hline
$(1,9)$ & $-3128874$ & 4\\ \hline
$(5,11)$ & $-168706650$ & 4\\ \hline
$(5,13)$ & $-541943850$ & 4\\ \hline
$(12,13)$ & 3121596576 & 4\\ \hline
$(13,16)$ & $6060449850$ & 4\\ \hline
$(1,40)$ & $-24552945606$ & 5\\ \hline
\end{tabular}
\caption{Rank of $E_A^1$ for some values of $r$ and $s$}
\end{table}

%%%%%%%%%%%%%%%%%%%%%%%%%%%%%%%%%%%%%%%%%%%%%

\subsection{A subfamily from an elliptic curve with positive rank}

We get another subfamily of $E_A$ by working on the previously obtained 
subfamily $E_A^1$. By varying $r$ and $s$ in $E_A^1$, we make an observation 
that many curves with rank 3 appear when $r$ is divisible by 4. Let $r = 4r_1$. 
Then we have 
\[ A = -96r_1^2 \left( 8r_1^2 - s^2 \right)\left( 16r_1^2 + s^2 \right)\left( 32r_1^2 - s^2 \right). \]
Consider the quartic equation 
\begin{equation}\label{subsec5.2}
	N^2 = 16r_1^2\left( 8r_1^2 - s^2 \right) \left( 16r_1^2 + s^2 \right) \left[ M^4 - 6^2\left( 32r_1^2 - s^2 \right)^2 e^4 \right].
\end{equation}
When $e = 1$ and $M = 12r_1$ we get 
$N^2 = 576r_1^2 \left( 8r_1^2 - s^2 \right)^2 \left( 16r_1^2 + s^2 \right) \left( -56r_1^2 + s^2 \right)$. 
Let us choose $r_1$ and $s$ so that 
$y^2 = \left( 16r_1^2 + s^2 \right) \left( -56r_1^2 + s^2 \right)$ has a 
solution in $\mathbb{Z}$. Equivalently, we choose 
$t = \tfrac{r_1}{s} \in \mathbb{Q}$ so that $y_1^2 = -896t^4 - 40t^2 + 1$ 
has a solution. This existence of a nontrivial rational solution 
$\left( \tfrac{2}{15}, \tfrac{17}{225} \right)$ tells us that this quartic curve 
is birationally equivalent to an elliptic curve (in Weierstrass form). Through 
SAGE we can compute the Jacobian of this elliptic curve and check that this 
curve has positive rank. This assures us the infinitude of rational numbers 
$t$ satisfying $y_1^2 = -896t^4 - 40t^2 + 1$. Consequently, we get 
infinitely many pairs of relatively prime integers $(r_1, s)$ satisfying the 
Diophantine equation 
$y^2 = \left( 16r_1^2 + s^2 \right) \left( -56r_1^2 + s^2 \right)$. 

If we screen our choice of $r_1$ and $s$ as above, we get that 
$16r_1^2 \left( 8r_1^2 - s^2 \right) \left( 16r_1^2 + s^2 \right)$ is an 
element of $\alpha(E_A^1(\mathbb{Q}))$ and the corresponding quadruple 
to \eqref{subsec5.1} is given by 
\[ (b_1, N, e, M) = \left( 16r_1^2 \left( 8r_1^2 - s^2 \right) \left( 16r_1^2 + s^2 \right), 24r_1\left( 8r_1^2 - s^2 \right)y, 1, 12r_1 \right) \] 
or (up to squares) 
\[ (b_1, N, e, M) = \left( \left( 8r_1^2 - s^2 \right) \left( 16r_1^2 + s^2 \right), 6r_1\left( 8r_1^2 - s^2 \right)y, 1, 192r_1^2 \right) \]  
Because $\alpha(E_A^1(\mathbb{Q}))$ is a multiplicative group, we can 
obtain a partial list of elements of $\alpha(E_A^1(\mathbb{Q}))$.

\begin{footnotesize}
\begin{table}[H]
\renewcommand{\arraystretch}{1.2}
\begin{tabular}{c | c | c}
In terms of $u$ and $v$ & In terms of $r_1$ and $s$ & New elements\\ \hline 
$\pm 1$ & $\pm 1$ & $\pm \left( 8r_1^2 - s^2 \right) \left( 16r_1^2 + s^2 \right)$\\
$\pm uv \left( u^2 - v^2 \right)$ & $\mp 6r_1^2 \left( 8r_1^2 - s^2 \right)\left( 16r_1^2 + s^2 \right)\left( 32r_1^2 - s^2 \right)$ &  $\mp 6 \left( 32r_1^2 - s^2 \right)$\\
$\pm uv$ & $\pm \left( 32r_1^2 - s^2 \right) \left( 16r_1^2 + s^2 \right)$ & $\pm \left( 8r_1^2 - s^2 \right) \left( 32r_1^2 - s^2 \right)$\\
$\pm \left( u^2 - v^2 \right)$ & $\mp 6 \left( 16r_1^2 + s^2 \right)$ & $\mp 6 \left( 8r_1^2 - s^2 \right)$\\
$\pm u(u + v)$ & $\pm 3 \left( 16r_1^2 + s^2 \right)$ & $ \pm 3 \left( 8r_1^2 - s^2 \right)$\\
$\pm v(u + v)$ & $\pm 3 \left( 32r_1^2 - s^2 \right)$ & $\pm 3 \left( 8r_1^2 - s^2 \right) \left( 16r_1^2 + s^2 \right) \left( 32r_1^2 - s^2 \right)$\\
$\pm u(u - v) $ & $\mp 2 \left( 8r_1^2 - s^2 \right) \left( 16r_1^2 + s^2 \right)$ & $\mp 2$\\
$\pm v(u - v)$ & $\mp 2 \left( 8r_1^2 - s^2 \right) \left( 32r_1^2 - s^2 \right)$ & $\mp 2 \left( 16r_1^2 + s^2 \right) \left( 32r_1^2 - s^2 \right)$
\end{tabular}
\caption{New elements (up to squares) of $\alpha(E_A^1(\mathbb{Q}))$ obtained 
after multiplying $\left( 8r_1^2 - s^2 \right) \left( 16r_1^2 + s^2 \right)$ to known 
elements of $\alpha(E_A^1(\mathbb{Q}))$ (entries of the second column)}
\label{table2}
\end{table}
\end{footnotesize}

We claim that the subfamily
\[E_{A}^2\,/\,\mathbb{Q}: y^2 = x^3 - n^2 x, \text{ with }
	A = -96r_1^2 \left( 8r_1^2 - s^2 \right)\left( 16r_1^2 + s^2 \right)\left( 32r_1^2 - s^2 \right) \] 
of $E_A$ has rank at least 3 when $r_1$ and $s$ are taken so that the 
following Diophantine equation is satisfied:
\[ y^2 = \left( 16r_1^2 + s^2 \right) \left( -56r_1^2 + s^2 \right). \]

In Subsection 5.1, to show that $E_A^1$ has rank at least 2, we showed that 
the entries in the first column of Table \ref{table2} are distinct elements of 
$\alpha(E_A^1(\mathbb{Q}))$. Now we have 32 known elements in 
$\alpha(E_A^2(\mathbb{Q}))$. One can proceed as we did before but the 
process will be quite tedious. To prove that the subfamily $E_A^2$ indeed has 
rank 3, we will use the Silverman's specialization theorem 
(cf. \cite{silverman}, Theorem 20.3, C.20). As an overview, the specialization 
theorem requires $m$ points of $E_n$, 
say $P_1, \hdots, P_m$, over a function field, in our case $\mathbb{Q}(t)$. 
If we choose a specific value of $t \in \mathbb{Q}$ so that the resulting 
points $P_1, \hdots, P_m$ over $\mathbb{Q}$ are of infinite order and are 
linearly independent, then the specialization theorem tells us that this 
property extends for all but finitely many $t \in \mathbb{Q}$. Consequently, 
$E_A^2$ has rank at least $m$.

\vspace*{0.05in}

In view of Proposition \ref{eltstopts}, we can extract the following points in 
$E_A^2(\mathbb{Q})$: 
\begin{align*} 
	P = (x_1, y_1), &\text{ since } v(u + v) \in \alpha(E_n^2(\mathbb{Q})),\\ 
	Q = (x_2, y_2), &\text{ since } \left( 8r_1^2 - s^2 \right) \left( 16r_1^2 + s^2 \right) \in 
\alpha(E_n^2(\mathbb{Q})), \text{ and}\\
	R = \psi(S) = (x_3^*, y_3^*), &\text{ since } 1 \in \overline{\alpha}(\overline{E_n^2}(\mathbb{Q})) \text{ (here, $S = (x_3, y_3) \in \overline{E_n^2}(\mathbb{Q})$)}
\end{align*}
with
\begin{align*}
x_1 &= 393216r^8 + 36864r^6s^2 - 48r^2s^6\\
y_1 &= 150994944r^{11}s + 9437184r^9s^3 - 442368r^7s^5 - 18432r^5s^7 + 576r^3s^9\\
%x1 = 393216*r^8 + 36864*s^2*r^6 - 48*s^6*r^2
%y1 = 150994944*s*r^11 + 9437184*s^3*r^9 - 442368*s^5*r^7 - 18432*s^7*r^5 + 576*s^9*r^3
x_2 &=  294912r^8 - 18432r^6s^2 - 2304r^4s^4\\
y_2 &= \left( 4718592r^{10} - 884736r^8s^2 + 4608r^4s^6 \right)y\\
%x2 =  294912*r^8 - 18432*s^2*r^6 - 2304*s^4*r^4
%y2 = (4718592*r^10 - 884736*s^2*r^8 + 4608*s^6*r^4)*y
x_3 &=  589824r^8 - 147456r^6s^2 + 9216r^4s^4\\
y_3 &= 754974720r^{12} - 207618048r^{10}s^2 + 17694720r^8s^4 - 589824r^6s^6 + 18432r^4s^8\\
%x3 = 589824*r^8 - 147456*s^2*r^6 + 9216*s^4*r^4
%y3 = 754974720*r^12 - 207618048*s^2*r^10 + 17694720*s^4*r^8 - 589824*s^6*r^6 + 18432*s^8*r^4
x_3^* &= 409600r^8 - 20480r^6s^2 + 1536r^4s^4 - 32r^2s^6 + s^8\\
y_3^* &= -73400320r^{12} - 32243712r^{10}s^2 + 2703360r^8s^4\\
	&\ \ \,\; - 81920r^6s^6 + 1920r^4s^8 + 48r^2s^{10} - s^{12}.
%x3^* = 409600*r^8 - 20480*s^2*r^6 + 1536*s^4*r^4 - 32*s^6*r^2 + s^8
%y3^* = -73400320*r^12 - 32243712*s^2*r^10 + 2703360*s^4*r^8 - 81920*s^6*r^6 + 1920*s^8*r^4 + 48*s^10*r^2 - s^12
\end{align*}
Notice that the $x$-coordinates of $P$, $Q$, and $R$ have degree 8 
while the $y$-coordinates have degree 12. They satisfy the equation 
$y^2 = x^3 - A^2 x$ with 
\[ A = -96r_1^2 \left( 8r_1^2 - s^2 \right)\left( 16r_1^2 + s^2 \right)\left( 32r_1^2 - s^2 \right), \] 
and both sides show a polynomial in $\mathbb{Z}[r_1, s]$ of degree 24. 
To move to $\mathbb{Q}(t)$, we can simply divide both sides by $s^{24}$ 
and make the appropriate change of coordinates: 
$x \mapsto \tfrac{x}{s^8}$, $y \mapsto \tfrac{y}{12}$, 
$A \mapsto \tfrac{A}{s^8}$ then put $t = \tfrac{r_1}{s}$. 
This results to a new elliptic curve 
\[ E_A^3\,/\,\mathbb{Q}: y^2 = x \left[ x^2 -  9216 t^4 \left( 8t^2 - 1 \right)^2 \left( 16t^2 + 1 \right)^2 \left( 32t^2 - 1 \right)^2 \right] \]
which is birationally equivalent to $E_A^2\,/\,\mathbb{Q}$. Moreover, the 
three points from above transforms into: 
\begin{align*}
	P &= \left( 393216t^8 + 36864t^6 - 48t^2, 150994944t^{11} + 9437184t^9\right.\\
		&\ \ \ \,\left. - 442368t^7 - 18432t^5 + 576t^3 \right),\\
	Q &= \left( 294912t^8 - 18432t^6 - 2304t^4, \left( 4718592t^{10} - 884736t^8 + 4608t^4 \right)y_1 \right), \text{ and}\\
	R &= \left( 409600t^8 - 20480t^6 + 1536t^4 - 32t^2 + 1, \right.\\
&\ \ \ \ \left. -73400320t^{12} - 32243712t^{10} + 2703360t^8 - 81920t^6 + 1920t^4 + 48t^2 - 1 \right),
\end{align*}
which all lie in $E_A^3(\mathbb{Q})$.

\begin{lemma}
Let $E_A^3\,/\,\mathbb{Q}$ be the elliptic curve discussed above. Choose 
$t \in \mathbb{Q}$ so that there exists a $y_1 \in \mathbb{Q}$ satisfying 
$y_1^2 = -896t^4 - 40t^2 + 1$. Then $E_A^3$ has rank at least 3 over 
$\mathbb{Q}$. 
\end{lemma}
\begin{proof}
Suppose $P$, $Q$, and $R$ are independent points in $E_A^3(\mathbb{Q}(t))$, 
of infinite order. Note that we can make this assumption about $P$, $Q$, and $R$ 
over $\mathbb{Q}(t)$ otherwise, there will exist no specialization of $t$ for which 
the points $P$ and $Q$ are linearly independent points of infinite order over 
$\mathbb{Q}$ to begin with. 

As $\left( \tfrac{2}{15}, \tfrac{17}{225} \right)$ is a point on the quartic, 
we can specialize at $t = \tfrac{2}{15}$. So we get the elliptic curve
\[ E : y^2 = x \left[ x^2 - \left( \tfrac{692527232}{854296875} \right)^2 \right], \]
which has the points:
\begin{align*}
	P &= \left( -\tfrac{518498368}{854296875}, \tfrac{402354733568}{961083984375} \right),\\
	Q &= \left( -\tfrac{228462592}{284765625}, \tfrac{1499171528704}{14416259765625} \right), \text{ and }\\
	R &= \left( \tfrac{86359849}{102515625}, \tfrac{5457760750771}{25949267578125} \right).
\end{align*}

Using SAGE, we can verify that $P$, $Q$, and $R$ are points of infinite order 
in $E(\mathbb{Q})$. All that is left is to check is if these three points are 
linearly independent. Again using SAGE, we compute the determinant of the 
height-pairing matrix of the points $P$, $Q$, and $R$. The determinant is 
approximately 473.22. It being nonzero indicates the linear independence of 
the points $P$, $Q$, and $R$. Finally, we use the Specialization Theorem to 
extend this result for all but finitely many values  $t \in \mathbb{Q}$.
\end{proof}

Since $E_A^2$ and $E_A^3$ are birationally equivalent elliptic curves over 
$\mathbb{Q}$, it follows that $\rank(E_A^2(\mathbb{Q})) \geq 3$. So we 
get another subfamily of $E_A$ with rank at least 3.

%%%%%%%%%%%%%%%%%%%%%%%%%%%%%%%%%%%%%%%%%%%%%

\section{Conclusion}

By the end of Section 4, we have already established the truth of the theorem. 

\mainresult*

\noindent The above result adds to the long list of partial answers given about 
the millenial-old congruent number problem. 
\begin{figure}[H]\label{fig1}
\begin{tikzpicture}
\node[anchor = north] (A) at (0,0) {\fbox{$n$ is a congruent number}}; 
\node[anchor = north] (B) at (8, 2) {\fbox{$nm^2 = uv \left( u^2 - v^2 \right)$ has a solution over $\mathbb{Z}$}};
\node[anchor = north] (C) at (8, 1) {\fbox{\eqref{eq4.1} has a solution over $\mathbb{Z}$}};
\node[anchor = north] (D) at (8,0) {\fbox{$\rank(E_n) \geq 1$}}; 
\node[anchor = north] (E) at (8, -1) {\fbox{\eqref{eq1} has a solution over $\mathbb{Q}$}};
\node[anchor = north] (F) at (8, -2) {\fbox{$2A_n = B_n$ if $n$ is odd; $2C_n = D_n$ if $n$ is even}};
\draw[<->, ,>=stealth] (A) -- (B);
\draw[<->, ,>=stealth] (A) -- (C);
\draw[<->, ,>=stealth] (A) -- (D);
\draw[<->, ,>=stealth] (A) -- (E);
\draw[->, ,>=stealth] (A) -- (F);
\end{tikzpicture}
\caption{Partial results for the congruent number problem}
\end{figure}
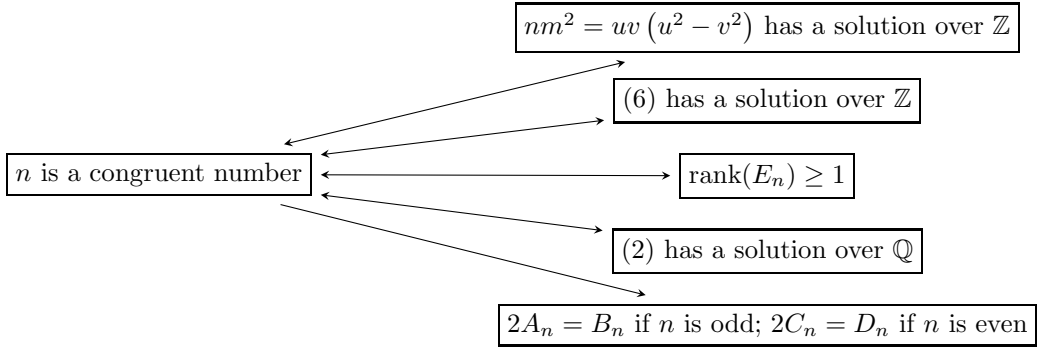

One particular problem that Theorem \ref{sec4thm1} gives a complete answer to 
is the question of finding an expression that generates the set of all congruent 
numbers if it exists. Theorem \ref{sec4thm1} answers in the affirmative and says 
that up to squares, the expression we are looking for is $uv\left( u^2 - v^2 \right)$, 
where $u$ and $v$ are integers. We also showed that a proof of Theorem 
\ref{sec4thm1} with 2-descent (see Proposition \ref{sec5prop1} and Proposition 
\ref{sec5prop2}) serves as a stepping stone for the construction of congruent 
number elliptic curves of rank higher than 1. 

The result of this paper targets specific questions about congruent numbers and 
elliptic curves. However, like all partial answers, it has its limitations. One particular 
weakness of this result is that the solvability of the equation 
$nm^2 = uv \left( u^2 - v^2 \right)$ over $\mathbb{Z}$ is not trivial. So for 
checking if a specific $n > 0$ is a congruent number, Theorem \ref{sec4thm1} may 
not be the best choice.

%%%%%%%%%%%%%%%%%%%%%%%%%%%%%%%%%%%%%%%%%%%%%

\end{document}